\documentclass[11pt]{article}

\usepackage{amsfonts}
\usepackage{amssymb,amsmath}
\usepackage{amsthm}
\usepackage{latexsym}
\usepackage{graphics}
\usepackage{epic}
\usepackage{epsfig}
\usepackage{psfrag}
\usepackage{bbm}
\usepackage{dsfont}
\usepackage{enumitem}
\usepackage{stmaryrd}
\usepackage{enumerate}
\usepackage{color}
\usepackage{hyperref}

\numberwithin{equation}{section}
\def\PP{\mathbb{P}}
\def\QQ{\mathbb{Q}}
\def\RR{\mathbb{R}}

\def\EE{\mathbb{E}}
\def\11{\mathbbm{1}}

\def\E{\mathbb{E}}
\def\P{\mathbb{P}}

\def\Q{\mathbb{Q}}
\def\N{\mathbb{N}}

\def\d{\partial}

\def\tT{\widetilde{T}}
\def\tX{\widetilde{X}}

\def\cE{{\cal E}}

\newtheorem{thm}{Theorem}[section]

\newtheorem{prop}[thm]{Proposition}

\theoremstyle{remark}

\begin{document}

\title{Exponential convergence to quasi-stationary distribution for absorbed one-dimensional diffusions with killing}

%short title : QSD of one dimensional diffusions with killing
%web page : http://www.normalesup.org/~villemonais/

\author{Nicolas Champagnat$^{1,2,3}$, Denis Villemonais$^{1,2,3}$}

\footnotetext[1]{IECL, Universit\'e de Lorraine, Site de Nancy, B.P. 70239, F-54506 Vandœuvre-lès-Nancy Cedex, France}
\footnotetext[2]{CNRS, IECL, UMR 7502, Vand{\oe}uvre-l\`es-Nancy, F-54506, France}  
\footnotetext[3]{Inria, TOSCA team, Villers-l\`es-Nancy, F-54600, France.\\
  E-mail: Nicolas.Champagnat@inria.fr, Denis.Villemonais@univ-lorraine.fr}

\maketitle

\begin{abstract}
  This article studies the quasi-stationary behavior of absorbed one-dimensional diffusion processes with killing on $[0,\infty)$.
  We obtain criteria for the exponential convergence to a unique quasi-stationary distribution in total variation, uniformly with
  respect to the initial distribution. Our approach is based on probabilistic and coupling methods, contrary to the classical
  approach based on spectral theory results. Our general criteria apply in the case where $\infty$ is entrance and 0 either regular
  or exit, and are proved to be satisfied under several explicit assumptions expressed only in terms of the speed and killing
  measures. We also obtain exponential ergodicity results on the $Q$-process. We provide several examples and extensions, including
  diffusions with singular speed and killing measures, general models of population dynamics, drifted Brownian motions and some
  one-dimensional processes with jumps.
\end{abstract}

\noindent\textit{Keywords:} {diffusions; one-dimensional diffusions with killing; absorbed process; quasi-stationary distribution;
  $Q$-process; uniform exponential mixing property; one dimensional processes with jumps.}

\medskip\noindent\textit{2010 Mathematics Subject Classification.} Primary: {60J60; 60J70; 37A25; 60B10; 60F99}. Secondary: {60G44;
  60J75}.

\section{Introduction}
\label{sec:intro}

This article studies the quasi-stationary behavior of general one-di\-men\-sio\-nal diffusion processes with killing in an interval
$E$ of $\RR$, absorbed at its finite boundaries. When the process is killed or absorbed, it is sent to some cemetary point $\d$. This
covers the case of solutions to one-dimensional stochastic differential equations (SDE) with space-dependent killing rate, but also
of diffusions with singular speed and killing measures.

We recall that a \emph{quasi-stationary distribution} for a continuous-time Markov process $(X_t,t\geq 0)$ on the state space
$E\cup\{\d\}$, is a probability measure $\alpha$ on $E$ such that
$$
\PP_\alpha(X_t\in\cdot\mid t<\tau_\partial)=\alpha(\cdot),\quad \forall t\geq 0,
$$
where $\PP_\alpha$ denotes the distribution of the process $X$ given that $X_0$ has distribution $\alpha$, and
$$
\tau_\d:=\inf\{t\geq 0: X_t=\d\}.
$$
We refer to~\cite{meleard-villemonais-12,vanDoorn2013,pollett-11} for general introductions to the topic.

Our goal is to give conditions ensuring the existence of a unique \emph{quasi-limiting distribution} $\alpha$ on $E$, i.e.\ a
probability measure $\alpha$ such that for all probability measures $\mu$ on $E$ such that $\mu(\mathring{E})>0$ and all $A\subset E$ measurable,
\begin{align}
\label{eq:QLD}
\lim_{t\rightarrow+\infty}\PP_\mu(X_t\in A\mid t<\tau_\partial)=\alpha(A),
\end{align}
where, in addition, the convergence is exponential and uniform with respect to $\mu$ and $A$. In particular, $\alpha$ is the unique
\emph{quasi-stationary distribution}.

This topic has been extensively studied for one-dimensional diffusions without killing
in~\cite{CCLMMS09,littin-12,miura-14,champagnat-villemonais-15b} \cite{collet-martinez-al-13b}, where nearly optimal criteria are obtained. The case with killing
is more complex and the existing results cover less general situations~\cite{Steinsaltz2007,kolb-steinsaltz-12}. In particular, these
references are restricted to the study of solutions to SDEs with continuous absorption rate up to the boundary of $E$, and let the questions of
uniqueness of the quasi-stationary distribution and of convergence in~\eqref{eq:QLD} open.

The present paper is focused on the case of a diffusion on $E=[0,+\infty)$ absorbed at 0 with scale function $s$, speed measure $m$ and killing
measure $k$, assuming that killing corresponds to an immediate jump to $\d=0$. We consider the situation where $\infty$ is an
entrance boundary and 0 is either exit or regular. Our results easily extend to cases of bounded intervals with reachable boundaries.

We give two criteria, each of them involving one condition concerning the diffusion (without killing) with scale function $s$ and speed measure
$m$, and another condition on the killing time. The condition on the diffusion without killing comes
from~\cite{champagnat-villemonais-15b} and covers nearly all one-dimensional diffusions on $[0,\infty)$ such that $\infty$ is an
entrance boundary and a.s.\ absorbed in finite time at $0$. The conditions on the killing time only concern the behavior of the
diffusion and of the killing measure in the neighborhood of 0, as soon as $\infty$ is an entrance boundary. In order to apply these
results to practical situations, we provide several explicit criteria ensuring all these conditions and a series of examples that
enter our setting.

This contribution improves known results in several directions. First, it covers situations of diffusion processes which
are not solutions to SDEs and cases of irregular and unbounded killing rates. Secondly, aside from proving the existence of a
quasi-stationary distribution, we also obtain exponential convergence of conditional distributions. In particular, all the initial
distributions of $X$ belong to the domain of attraction of the quasi-stationary distribution. This is of great importance in
applications, since in practice, one usually has no precise estimate of the initial distribution. In addition, our estimates of
convergence are exponential and uniform in total variation norm, and hence provide a uniform bound for the time needed to observe
stabilisation of the conditional distribution of the process, regardless of the initial distribution (see also the discussion
in~\cite[Ex.\,2]{meleard-villemonais-12}).

The main methodological novelty of our proofs relies on its purely probabilistic approach. We do not use any spectral theoretical
result (typically, Sturm-Liouville theory), which are the key tool of all the previously cited works on diffusions,
except~\cite{champagnat-villemonais-15b}. Instead, we use criteria for general Markov processes proved
in~\cite{champagnat-villemonais-15}, based on coupling and Dobrushin coefficient techniques. These criteria also imply exponential
ergodicity for the $Q$-process, defined as the process $X$ conditioned to never be absorbed. The generality and flexibility of this
approach also allow to cover, without substantial modification of the arguments, many situations where the spectral theory has
received much less attention, such as one-dimensional diffusions with jumps and killing.

For our study, we need to give a probabilistic formulation of the property that $\infty$ is entrance. In the case without killing,
this is known as the classical property of \emph{coming down from infinity} (see e.g.~\cite{CCLMMS09,champagnat-villemonais-15b}).
For a diffusion with killing, we show that a process with entrance boundary at $\infty$ \emph{comes down from infinity before
  killing} in the sense that the diffusion started from $+\infty$ hits 0 in finite time before killing with positive probability.

The paper is organized as follows. In Section~\ref{sec:construction}, we precisely define the absorbed diffusion processes with
killing under study and give their construction as time-changed Brownian motions. Although quite natural, this alternative
construction of killed diffusions is not given in classical references. Section~\ref{sec:main-result} contains the statement of our
main results on the exponential convergence of conditional distributions, the asymptotic behavior of the probability of survival and
the existence and ergodicity of the $Q$-process. In Subsection~\ref{sec:C-C'}, we give explicit criteria ensuring the conditions of
Section~\ref{sec:main-result}. We focus in Sections~\ref{sec:construction} and~\ref{sec:QSD-diff-kill} on diffusions on natural
scale. The extension to general diffusions is given in Section~\ref{sec:ex1}, together with a series of examples of diffusions which
are not solutions to SDEs, of diffusions with unbounded or irregular killing rates and of one-dimensional processes with killing and
jumps. Section~\ref{sec:popu-dyn} is devoted to the study of models of population dynamics of the form
\begin{align*}
  dY_t=\sqrt{Y_t}d B_t+Y_t h(Y_t)dt
\end{align*}
and Section~\ref{sec:sturm-liouville} to drifted Brownian motions, which are the basic models
of~\cite{Steinsaltz2007,kolb-steinsaltz-12}. Section~\ref{sec:pties-diff} concerns the property of coming down from
infinity before killing. Finally, Section~\ref{sec:proof-results-kill} gives the proofs of our main results.

\section{Absorbed diffusion processes with killing}
\label{sec:construction}

Our goal is to construct diffusion processes with killing on $[0,+\infty)$, absorbed at $\partial=0$. The typical situation
corresponds to stochastic population dynamics of continuous densities with possible continuous or sudden extinction.

We recall that a stochastic process $(X_t,t\geq 0)$ on $[0,+\infty)$ is called a diffusion (without killing) if it has a.s.\
continuous paths, satisfies the strong Markov property and is \emph{regular}. By regular, we mean that for all $x\in(0,\infty)$ and
$y\in[0,\infty)$, $\P_x(T_y<\infty)>0$, where $T_y$ is the first hitting time of $y$ by the process $X$. Given such a process, there
exists a continuous and strictly increasing function $s$ on $[0,\infty)$, called the \emph{scale function}, such that $(s(X_{t\wedge
  T_0}),t\geq 0)$ is a local martingale~\cite{freedman-83}. The stochastic process $(s(X_t),t\geq 0)$ is itself a diffusion process
with identity scale function. Replacing $(X_t,t\geq 0)$ by $(s(X_t),t\geq 0)$, we can assume without loss of generality that
$s(x)=x$.

To such a process $X$ on natural scale, one can associate a unique locally finite positive measure $m(dx)$ on $(0,\infty)$, called
the \emph{speed measure} of $X$, which gives positive mass to any open subset of $(0,+\infty)$ and such that $X_t=B_{\sigma_t}$ for
all $t\geq 0$ for some standard Brownian motion $B$, where
\begin{align}
  \label{eq:As}
  \sigma_t=\inf\left\{s>0:A_s>t\right\}, \quad\text{with }A_s=\int_0^\infty L^x_s\, m(dx)
\end{align}
and $L^x$ is the local time of $B$ at level $x$. Conversely, given any positive locally finite measure $m$ on $(0,\infty)$ giving
positive mass to any open subset of $(0,\infty)$, any such time change of a Brownian motion defines a regular diffusion on
$[0,\infty)$~\cite[Thm.\,23.9]{kallenberg-02}. Note that, since $\sigma_t$ is continuous and since $X_t=B_{\sigma_t}$ for all $t\geq 0$ a.s., we have
\begin{equation}
  \label{eq:lien-tau-d-brownien}
  \sigma_{T_0}=T^B_0,  \quad\text{or, equivalently,}\quad T_0=A_{T^B_0},
\end{equation}
where $T^B_x$ is the first hitting time of $x\in\RR$ by the process $B$.

In this work, we study diffusion processes on $[0,+\infty)$ with killing as defined in~\cite{ito-mckean-74}. As above, we can assume
without loss of generality that the diffusion is on natural scale. We show below that such a process can be obtained from an explicit
pathwise construction from a given Brownian motion $(B_t,t\geq 0)$ and an independent exponential random variable $\mathcal{E}$ of
parameter 1. Although this construction is quite natural, this is not done in the classical
references~\cite{ito-mckean-74,freedman-83,meleard-86,kallenberg-02}.

Let $k$ and $m$ be two positive locally finite measures on $(0,+\infty)$, such that $m$ gives positive mass to any open subset of
$(0,+\infty)$. The measures $m$ and $k$ will be referred to as the speed measure and the killing measure, respectively, of the
diffusion with killing.

We first define the diffusion process on natural scale without killing as above by
$$
\tX_t:=B_{\sigma_t},\qquad \forall t\geq 0,
$$
where $\sigma_t$ is defined by~\eqref{eq:As}. Next, we define for all $t\geq 0$
\begin{equation}
  \label{eq:kappa_t}
  \kappa_t:=\int_0^\infty L^y_{\sigma_t}\,k(dy),\quad\forall t\geq 0.  
\end{equation}
It follows from the definition of the local time $(L^{\tX,x}_t,t\geq 0, x\geq 0)$ of the process $\tX$ of~\cite[p.\,160]{freedman-83}
that
$$
\kappa_t=\int_0^\infty L^{\tX,y}_t\, k(dy).
$$
Let $\mathcal{E}$ be an exponential r.v.\ of parameter 1 independent of $B$. Then we define
\begin{equation}
  \label{eq:def-tau-kappa}
  \tau_\kappa:=\inf\{t\geq 0:\kappa_t\geq \mathcal{E}\}\quad \text{and}\quad \tT_0:=\inf\{t\geq 0:\tX_t=0\}.
\end{equation}
Our goal is to construct a process $X$ that can hit 0 either continuously following the path of $\tX$, or discontinuously when it is
killed at time $\tau_\kappa$. This leads to the following definition of the \emph{time of discontinuous absorption} (or \emph{killing time}) $\tau^d_\d$ and the
 \emph{time of continuous absorption} $\tau^c_\d$:
\begin{equation}
  \label{eq:def-absorption-times}
  \tau^d_\d:=
  \begin{cases}
    \tau_\kappa & \text{if }\tau_\kappa<\tT_0, \\
    +\infty & \text{if }\tT_0\leq\tau_\kappa
  \end{cases}
  \qquad \text{and} \qquad
  \tau^c_\d:=
  \begin{cases}
    +\infty & \text{if }\tau_\kappa<\tT_0, \\
    \tT_0 & \text{if }\tT_0\leq\tau_\kappa.
  \end{cases}
\end{equation}
The absorption time $\tau_\d$ is then defined as
$$
\tau_\d:=\tau^d_\d\wedge\tau^c_\d.
$$
Then, the absorbed diffusion process with killing $X$ is defined as
\begin{equation}
  \label{eq:def-diff-with-killing}
  X_t:=
  \begin{cases}
    \tX_t & \text{if\ }t<\tau_\d, \\
    \d=0 & \text{otherwise.}
  \end{cases}
\end{equation}

\begin{prop}
  \label{prop:diff-kill-ito-mckean}
  The process $(X_t,t\geq 0)$ defined in~\eqref{eq:def-diff-with-killing} is a diffusion process with killing on natural scale, with
  speed measure $m$ and killing measure $k$, as defined in~\textup{\cite{ito-mckean-74}}.
\end{prop}

\begin{proof}
  %In~\cite{ito-mckean-74}, a diffusion process with killing is defined from a diffusion
  %process without killing, given in our case by the process $(\tX_t,t\geq
  %0)$.
  Note that
  $$
  \kappa_t=\int_0^\infty L^{X,y}_t\, k(dy),\quad\forall t<\tau_\kappa,
  $$
  since $X$ and $\tX$ coincide before $\tau_\kappa$. Since $E$ is independent of $B$,
  \begin{equation}
    \label{eq:calcul}
    \PP(\tau_\kappa>t\mid \tX_t,\,t\geq 0)=\PP(\tau_\kappa>t\mid B_s,\,s\geq 0)=\exp\left(-\int_0^\infty L^{X,y}_t\, k(dy)\right).    
  \end{equation}
  This is exactly the definition of the distribution of the killing time of the diffusion process with killing
  of~\cite[p.\,179]{ito-mckean-74}. 
\end{proof}

In particular, in the case where $k$ is absolutely continuous with respect to $m$, the killing rate of the diffusion at position $x$
(in the sense e.g.\ of~\cite{kolb-steinsaltz-12}) is given by $\frac{dk}{dm}(x)$. This can be deduced from~\eqref{eq:calcul} and from
the occupation time formula of~\cite{freedman-83}.

In addition, the property of regularity of the diffusion with killing can be easily checked from this definition, as shown in the
next proposition.

\begin{prop}
  \label{prop:regularity}
  For all $x,y>0$,
  $$
  \P_x(T_y<\infty)>0,
  $$
  and for all $x>0$ and $t>0$,
  \begin{equation}
    \label{eq:tau_d>t}
    \PP_x(\tau_\d>t)>0,\quad \forall x>0,\ \forall t\geq 0.
  \end{equation}
\end{prop}

\begin{proof}
  Fix $x\not= y$. On the event $\{\tT_y<\infty\}$, we have $\sigma_{\tT_y}=T^B_y<T^B_0$. In particular, the infimum of $B_s$ over
  $[0,T^B_y]$ is positive, and its maximum is finite. Hence the function $y\mapsto L^y_{T^B_y}$ is a.s.\ continuous with compact
  support in $(0,+\infty)$, hence a.s.\ bounded, on the event $\{\tT_y<\infty\}$. Since $k$ is locally finite on $(0,+\infty)$ and
  $\P_x(\tT_y<\infty)>0$, this implies that
  \begin{align*}
    \P_x(T_y<\infty) & =\P_x(\tT_y<\infty\text{ and }\kappa_{\tT_y}<{\cal E}) \\ & =
    \E_x\left[\mathbbm{1}_{\tT_y<\infty}\exp\left(-\int_0^\infty L^y_{T^B_y}\,k(dy)\right)\right]>0.
  \end{align*}

  To prove~\eqref{eq:tau_d>t}, we fix $t>0$. As above, a.s.\ on the event $\{t<\tau_\d^c\}=\{\sigma_t<T^B_0\}$, the function $y\mapsto
  L^y_{\sigma_t}$ is continuous with compact support on $(0,+\infty)$. Since $k$ is locally finite on $(0,+\infty)$, we deduce that  $\kappa_t<\infty$ a.s. Since in addition $\PP_x(\tX_t>0)>0$ for all $t>0$,~\eqref{eq:tau_d>t} follows.
\end{proof}

In the sequel, following the terminology of~\cite{feller-52}, we assume that $+\infty$ is an entrance boundary, or equivalently (for diffusion processes on natural scale)
\begin{equation}
  \label{eq:hyp-kill-infinity-entrance}
  \int_1^\infty y(dm(y)+dk(y))<\infty,
\end{equation}
and that $0$ is either regular or an exit point, or equivalently
\begin{equation}
  \label{eq:hyp-kill-0}
  \int_0^1 y(dm(y)+dk(y))<\infty.
\end{equation}
Note that, in the case where $k=0$, $X=\tX$ is a diffusion without killing. In this case, the assumption $\int_1^\infty
y\,dm(y)<\infty$ corresponds to the fact that $X$ comes down from infinity~\cite{champagnat-villemonais-15b}, i.e.\ that there exist
 $t>0$ and $y>0$ such that
$$
\inf_{x>y}\P_x(T_y<t)>0,
$$
and $\int_0^1 y\,dm(y)<\infty$ is equivalent to assuming that $\tau_\d=\tau^c_\d<\infty$ a.s.

Further general properties of diffusion processes with killing are studied in Section~\ref{sec:pties-diff}, related to the notion of
\emph{coming down from infinity}.

\section{Quasi-stationary distributions for diffusion processes with killing}
\label{sec:QSD-diff-kill}

We consider as above a diffusion process $X$ on $[0,+\infty)$ with killing, on natural scale, with speed measure $m(dx)$ and killing
measure $k(dx)$.

\subsection{Exponential convergence to quasi-stationary distribution}
\label{sec:main-result}

We provide here sufficient criteria ensuring the existence of a unique quasi-stationary distribution for $X$, with exponential
convergence of conditional distributions in total variation. Both criteria involve the diffusion process without killing $\tX$ of
Section~\ref{sec:construction}.

\bigskip\noindent\textbf{Condition (C)}
Assume that $\int_0^\infty y\,(m(dy)+k(dy))<\infty$ and that there exist two constants $t_1,A>0$ such that
\begin{equation}
  \label{eq:hyp-B-prime}
  \PP_x(t_1<\tT_0)\leq Ax,\quad\forall x>0
\end{equation}
and
\begin{align}
  \label{eq:assumption-C-Ax}
  \P_x(\tau_\d^d<\tau_\d^c)\leq Ax,\quad\forall x>0.
\end{align}

\bigskip\noindent\textbf{Condition (C')} Assume that $\int_0^\infty y\,(m(dy)+k(dy))<\infty$ and that there exist three constants
$t_1,A,\varepsilon>0$ such that
\begin{equation*}
  \PP_x(t_1<\tT_0)\leq Ax,\quad\forall x>0
\end{equation*}
and $k$ is absolutely continuous w.r.t. $m$ on $(0,\varepsilon)$ and satisfies
\begin{align}
  \label{eq:C'}
  \frac{d k}{d m}(x)\leq A, \quad \forall x\in(0,\varepsilon).
\end{align}

\bigskip We give practical criteria to check Conditions~(C) and~(C') is Subsection~\ref{sec:C-C'} and we give examples of
applications and compare with existing results in Section~\ref{sec:ex-bib}.

\begin{thm}
  \label{thm:QSD_full_with_killing}
  Assume that $X$ is a one-dimensional diffusion on natural scale with speed measure $m$ and killing measure $k$. If Assumption
  \textup{(C)} or Assumption \textup{(C')} is satisfied, then there exists a unique probability measure $\alpha$ on $(0,+\infty)$ and
  two constants $C,\gamma>0$ such that, for all initial distribution $\mu$ on $E$ such that $\mu(\{0\})<1$,
  \begin{align}
    \label{eq:expo-cv-thm-with-killing}
    \left\|\PP_\mu(X_t\in\cdot\mid t<\tau_\partial)-\alpha(\cdot)\right\|_{TV}\leq C e^{-\gamma t},\ \forall t\geq 0.
  \end{align}
  In this case, $\alpha$ is the unique quasi-stationary distribution for the process.
\end{thm}

%\begin{rem}
%  \label{rem:better-(C)}
%  As will appear in Step~1 of the proof of Section~\ref{sec:(C)->(A)}, Thm.~\ref{thm:QSD_full_with_killing} would also hold true
%  if~\eqref{eq:hyp-B-prime} and~\eqref{eq:assumption-C-Ax} in Condition~(C) are replaced by the weaker condition
%  $\PP_x(t_1\wedge\tau_\kappa<\tT_0)\leq Ax$ for all $x>0$. However, the explicit criteria given in Subsection~\ref{sec:C-C'}
%  directly entail Condition (C), so we stick to this stronger version for convenience.
%\end{rem}

We recall that the quasi-stationary distribution $\alpha$ of Thm.~\ref{thm:QSD_full_with_killing} satisfies the following classical
property: there exists $\lambda_0>0$ such that $\P_\alpha(t<\tau_\d)=e^{-\lambda_0 t}$. The same assumptions also entail the
following result on the asymptotic behavior of absorption probabilities.

\begin{prop}
  \label{prop:eta-with-killing}
  Assume that $X$ is a one-dimensional diffusion on natural scale with speed measure $m$ and killing measure $k$. If Assumption
  \textup{(C)} or Assumption \textup{(C')} is satisfied, then there exists a bounded function $\eta:(0,\infty)\rightarrow(0,\infty)$
  such that
  \begin{align}
    \label{eq:convergence-to-eta-with-killing}
    \eta(x)=\lim_{t\rightarrow\infty} \frac{\P_x(t<\tau_\d)}{\P_\alpha(t<\tau_\d)}=\lim_{t\rightarrow +\infty} e^{\lambda_0 t}\P_x(t<\tau_\d),
  \end{align}
  where the convergence holds for the uniform norm. Moreover, $\alpha(\eta)=1$, $\eta(x)\leq C x$ for all $x\geq 0$ and some constant
  $C>0$, and $\eta$ belongs to the domain of the infinitesimal generator $L$ of $X$ on the set of bounded measurable functions on
  $[0,+\infty)$ equipped with uniform norm and
  \begin{align}
    \label{eq:eigen-function-with-killing}
    L\eta=-\lambda_0\eta.
  \end{align}
\end{prop}

Still under the same assumptions, we also obtain the exponential ergodicity of the $Q$-process, which is defined as the process $X$
conditioned to never be absorbed.

\begin{thm}
  \label{thm:Q-proc-with-killing}
  Assume that $X$ is a one-dimensional diffusion on natural scale with speed measure $m$ and killing measure $k$. If Assumption
  \textup{(C)} or Assumption \textup{(C')} is satisfied, then the three following properties hold true.
  \begin{description}
  \item[\textmd{(i) Existence of the $Q$-process.}] There exists a family $(\QQ_x)_{x>0}$ of
    probability measures on $\Omega$ defined by
    $$
    \lim_{t\rightarrow+\infty}\PP_x(A\mid t<\tau_\partial)=\QQ_x(A)
    $$
    for all ${\cal F}_s$-measurable set $A$, for all $s\geq 0$. The process $(\Omega,({\cal F}_t)_{t\geq 0},(X_t)_{t\geq 0},(\QQ_x)_{x>0})$ is a
    $(0,\infty)$-valued homogeneous strong Markov process. 
  \item[\textmd{(ii) Transition kernel.}] The transition kernel of the Markov process $X$ under $(\QQ_x)_{x>0}$ is given by
    \begin{align*}
      \tilde{p}(x;t,dy)=e^{\lambda_0 t}\frac{\eta(y)}{\eta(x)}p(x;t,dy)\mathbbm{1}_{y>0},\quad\forall t\geq 0,\ x>0,
    \end{align*}
    where $p(x;t,dy)$ is the transition distribution of the diffusion with killing $X$. % In other words, for all
  \item[\textmd{(iii) Exponential ergodicity.}] The probability measure $\beta$ on $(0,+\infty)$ defined by
    \begin{align*}
      \beta(dx)=\eta(x)\alpha(dx).
    \end{align*}
    is the unique invariant distribution of $X$ under $\QQ$. Moreover, there exist positive constants $C,\gamma$ such that, for any
    initial distributions $\mu_1,\mu_2$ on $(0,+\infty)$,
    \begin{align*}
      \left\|\Q_{\mu_1}(X_t\in\cdot)-\Q_{\mu_2}(X_t\in\cdot)\right\|_{TV}\leq Ce^{-\gamma t}\|\mu_1-\mu_2\|_{TV},
    \end{align*}
    where $\Q_\mu=\int_{(0,\infty)} \Q_x\,\mu(dx)$.
  \end{description}
\end{thm}

The proofs of these results are given in Section~\ref{sec:proof-results-kill}.

\subsection{On Conditions~(C) and~(C')}
\label{sec:C-C'}

The first step to check Conditions~(C) or~(C') is to check~\eqref{eq:hyp-B-prime}. This condition is the one ensuring exponential
convergence of conditional distributions of general diffusion processes without killing~\cite[Thm.\,3.1]{champagnat-villemonais-15b},
and it has been extensively studied in~\cite{champagnat-villemonais-15b}. We recall here the most general criterion, which applies to
all practical cases where $\infty$ is entrance and $0$ is exit or natural, since it fails only when $m$ has strong oscillations near
0 (see~\cite[Rk.\,4]{champagnat-villemonais-15b}).

\begin{thm}[\cite{champagnat-villemonais-15b} Thm~3.7]
  \label{thm:h-transfo-3}
  Assume that $\int_0^\infty y\, m(dy)<\infty$ and
  \begin{equation}
    \label{eq:matsumoto}
    \int_0^1\frac{1}{x}\,\sup_{y\leq x}\left(\frac{1}{y}\int_{(0,y)} z^2\,m(dz)\right)\,dx<\infty.
  \end{equation}
  Then, for all
  $t>0$, there exists $A_t<\infty$ such that
  \begin{align*}
    \P_x(t<\tT_0)\leq A_t\, x,\ \forall x>0.
  \end{align*}
\end{thm}

Using this result, the remaining conditions in~(C) and~(C') only deal with the behavior of the diffusion process $X$ near 0. We
provide two sufficient criteria for~\eqref{eq:assumption-C-Ax} in Condition~(C) in Proposition~\ref{prop:diff-kill-C1} and
Proposition~\ref{prop:diff-kill-C2}. Note that checking property~\eqref{eq:C'} in Condition (C') from given measures $m$ and $k$ is
straightforward.

\begin{prop}
  \label{prop:diff-kill-C1}
  Assume that $\int_0^\infty y\,k(dy)<\infty$ and that there exist $\varepsilon>0$ and a constant $C>0$ such that $k$ is absolutely
  continuous w.r.t. the Lebesgue measure $\Lambda$ on $(0,\varepsilon)$ and such that
  \begin{align}
    \label{eq:1/x}
    \frac{d k}{d \Lambda}(x)\leq \frac{C}{x}, \ \forall x\in(0,\varepsilon).
  \end{align}
  Then there exists a positive constant $A>0$ such that
  \begin{align*}
    \P_x(\tau_\d^d<\tau_\d^c)\leq A x.
  \end{align*}
\end{prop}

Note that, contrary to Condition~(C'), the previous criterion does not require the killing rate $dk/dm$ of the diffusion process to
be bounded on a neighborhood of $0$. It is actually independent of the measure $m$. For instance, if $m(dx)=dx$ and
$k(dx)=\frac{1}{x}dx$ on a neighborhood of $0$, then the assumptions of Proposition~\ref{prop:diff-kill-C1} are satisfied while
$dk/dm(x)=1/x$ is unbounded near $0$.

\begin{proof}[Proof of Proposition~\ref{prop:diff-kill-C1}]
  We have from~\eqref{eq:kappa_t}
  \begin{align*}
    \P_x(\tau_\d^d<\tau_\d^c)=\P_x\left(\kappa_{\tT_0}>\cE\right) % =\P_x\left(\int_0^\infty L_{\tau_\d^c}^{X,y} dk(y)>\cE\right)\\
    =\P_x\left(\int_0^\infty L_{\sigma_{\tT_0}}^{y} dk(y)>\cE\right)
  \end{align*}
  Hence, by~\eqref{eq:lien-tau-d-brownien},
  \begin{align*}
    \P_x(\tau_\d^d<\tau_\d^c)= \P_x\left(\int_0^\infty L_{T^B_0}^{y} dk(y)>\cE\right)
    =\P_x^Y\left(T_0^Y>\cE\right),
  \end{align*}
  where $Y$ is a diffusion process on $(0,\infty)$ on natural scale with speed measure $k$, absorbed at $0$ and without killing. We have
  \begin{align*}
    \P_x^Y(T_0^Y>\cE)&\leq \P_x^Y(\cE<T_0^Y<T_\varepsilon^Y)+\P_x^Y(T_\varepsilon^Y\leq T_0^Y)\\
    &\leq  \P_x^Y(\cE<T_0^Y<T_\varepsilon^Y)+x/\varepsilon,
  \end{align*}
  since $Y$ is on natural scale and hence is a local martingale. As a consequence, it only remains to prove that
  $\P_x^Y(\cE<T_0^Y<T_\varepsilon^Y)\leq Ax$. Because of our assumption~\eqref{eq:1/x},
  \begin{align*}
    \P_x^Y(\cE<T_0^Y<T_\varepsilon^Y)&=\P_x\left(\left\{T_0^B<T_\varepsilon^B\right\}\cap \left\{\int_0^\infty L_{T^B_0}^{y} dk(y)>\cE\right\}\right)\\
    &\leq \P_x\left(\left\{T_0^B<T_\varepsilon^B\right\}\cap \left\{\int_0^\infty L_{T^B_0}^{y} \frac{dy}{y}>\cE/C\right\}\right)\\
    &\leq\P_x\left(\int_0^\infty L_{T^B_0}^{y} \frac{dy}{y}>\cE/C\right)=1-v(x)
  \end{align*}
  where $v(x)=\P_x^Z(T_0^Z\leq \cE/C)$ with $Z$ the diffusion on $(0,+\infty)$ on natural scale and with speed measure $(1/y) dy$,
  i.e.\ solution to the SDE
  \begin{align*}
    dZ_t=\sqrt{Z_t}dW_t,
  \end{align*} 
  for $(W_t)_{t\geq 0}$ a standard Brownian motion. The diffusion process $(Z_t)_{t\geq 0}$ is a continuous state branching process
  (Feller diffusion), hence by the branching property,
  \begin{align*}
    v(x+y)=v(x)v(y),\ \forall x,y \geq 0.
  \end{align*}
  Since $v(x)$ is non-increasing and $v(0)=1$, we deduce from standard arguments that there exists a constant $A>0$ such that
  \begin{align*}
    v(x)=e^{-Ax}.
  \end{align*}
  Proposition~\ref{prop:diff-kill-C1} follows.
\end{proof}

\begin{prop}
  \label{prop:diff-kill-C2}
  Assume that $\int_1^\infty yk(dy)<\infty$ and $k([0,1])<\infty$. Then there exists a positive constant $A>0$ such that
  \begin{align*}
    \P_x( \tau_\d^d<\tau_\d^c)\leq A x.
  \end{align*}
\end{prop}

We emphasize that measures $k$ which are not absolutely continuous w.r.t.\ the Lebesgue measure are of course also covered by
Prop.~\ref{prop:diff-kill-C2}. Note also that Prop.~\ref{prop:diff-kill-C2} does not imply Prop.~\ref{prop:diff-kill-C1} since the
killing measure $dk(x)=(1/x)d\Lambda(x)\mathbbm{1}_{x\in (0,1)}$ is not finite.

\begin{proof}[Proof of Proposition~\ref{prop:diff-kill-C2}]
  As in the beginning of the proof of Proposition~\ref{prop:diff-kill-C1}, we obtain
  \begin{align*}
    \P_x(\tau_\d^c<\tau_\d^d)&= \P_x\left(\int_0^\infty L_{T^B_0}^{y} dk(y)<\cE\right)\\
    &=\E_x\left[\exp\left(-\int_0^\infty L_{T^B_0}^{y} dk(y)\right)\right]\\
    &\geq 1-\E_x\left(\int_0^\infty L_{T^B_0}^{y} dk(y)\right)\\
    &=1-\E_x^Y(T_0\mid Y_0=x),
  \end{align*}
  where $\E^Y$ denotes the expectation w.r.t. the law of a diffusion process $Y$ on $(0,+\infty)$ on natural scale, without killing
  and whose speed measure is $k(dy)$. Finally, using classical Green formula for diffusion processes without killing % (see
  applied to $Y$, we conclude that
  \begin{equation*}
    \P_x(\tau_\d^c<\tau_\d^d)\geq 1-\int_0^\infty (x\wedge y)\,k(dy)\geq 1-xk([0,1])-x\int_1^\infty y\,k(dy)\geq 1-Cx. \qedhere
  \end{equation*}
\end{proof}

\section{Examples and comparison with existing results}
\label{sec:ex-bib}

\subsection{On general diffusions}
\label{sec:ex1}

Let us first recall that our results also cover the case of general (i.e.\ not necessarily on natural scale) killed diffusion
processes $(Y_t,t\geq 0)$ on $[0,\infty)$ with speed measure $m_Y$ and killing measure $k_Y$, such that the diffusion without killing
$\widetilde{Y}$ hits 0 in a.s.\ finite time. Under these assumptions, there exists a continuous and strictly increasing scale
function $s:[0,\infty)\rightarrow[0,\infty)$ of the process $Y$ such that $s(0)=0$ and $s(\infty)=\infty$, and our results apply to
the process $X_t=s(Y_t)$ on natural scale, whose speed measure and killing measure are respectively given by
\begin{align}
  \label{eq:k_X-m_X}
  m_X(dx)=m_Y*s(dx)\quad\text{and}\quad k_X(dx)=k_Y*s(dx),
\end{align}
where $m_Y*s$ and $k_Y*s$ denote the pushforward measures of $m_Y$ and $k_Y$ through the function $s$ (the formula for $m_X$ can be
found in~\cite[Thm.\,VII.3.6]{Revuz1999} and for the measure $k_X$, one may use for example~\eqref{eq:kappa_t} and the fact that
$L^{y,\widetilde{Y}}_t=L^{s(y),\tX}_t$ which follows from the occupation time formula of~\cite[p.\,160]{freedman-83}).

Hence $X$ satisfies (C) or (C') if and only if $Y$ satisfies the following Conditions~(D) or (D') respectively. In particular, the
conclusions of Theorem~\ref{thm:QSD_full_with_killing} apply to $Y$ if (D) or (D') is fulfilled.

\bigskip\noindent\textbf{Condition (D)} Assume that $\int_0^\infty s(y)\,(m_Y(dy)+k_Y(dy))<\infty$. Assume also that there exist two
constants $t_1,A>0$ such that
\begin{equation*}
  \PP(t_1<\tT_0\mid Y_0=y)\leq As(y),\quad\forall y>0
\end{equation*}
and
\begin{align*}
  \P(\tau_\d^{d}<\tau_\d^{c}\mid Y_0=y)\leq As(y) ,\quad\forall y>0,
\end{align*}
where the absorption times $\tT_0$, $\tau_\d^{d}$ and $\tau_\d^c$ are constructed here from the processes $Y$ and $\widetilde{Y}$.
\bigskip

\bigskip\noindent\textbf{Condition (D')} Assume that $\int_0^\infty s(y)\,(m_Y(dy)+k_Y(dy))<\infty$. Assume also that there exist
three constants $t_1,A,\varepsilon>0$ such that
\begin{equation*}
  \PP(t_1<\tT_0\mid Y_0=y)\leq As(y),\quad\forall y>0
\end{equation*}
and $k_Y$ is absolutely continuous w.r.t. $m_Y$ on $(0,\varepsilon)$ and satisfies
\begin{align*}
  \frac{d k_Y}{d m_Y}(y)\leq A, \quad \forall y\in(0,\varepsilon).
\end{align*}

In particular, we deduce from~\eqref{eq:k_X-m_X} the following criteria for (D) and (D').
\begin{prop}
  \label{prop:(D)-(D')}
  Assume that $\int_0^\infty s(y)\,(m_Y(dy)+k_Y(dy))<\infty$ and that there exist two constants $t_1,A>0$ such that
  \begin{equation*}
    \PP(t_1<\tT_0\mid Y_0=y)\leq As(y),\quad\forall y>0.
  \end{equation*}
  Then
  \begin{description}
  \item[\textmd{(i)}] If $k_Y([0,1])<\infty$, then $Y$ satisfies (D).
  \item[\textmd{(ii)}] If there exists $\varepsilon,\rho,C>0$ such that 
    \begin{align*}
      \frac{dk_Y}{d\Lambda}(y)\leq \frac{C}{s(y)},\ \forall y\in (0,\varepsilon)
    \end{align*}
    and
    \begin{equation}
      \label{eq:hyp-ds}
      s(y)-s(x)\geq \rho(y-x),\quad\forall 0<x<y<\varepsilon,
    \end{equation}
    then $Y$ satisfies (D).
  \item[\textmd{(iii)}] If the killing rate $\kappa_Y(y):=\frac{dk_Y}{dm_Y}(y)$ exists and is uniformly bounded on
    $(0,\varepsilon)$ for some $\varepsilon>0$, then $Y$ satisfies (D').
  \end{description}
\end{prop}

\begin{proof}
  Point~(i) follows immediately from Prop.~\ref{prop:diff-kill-C2} since $k_X([0,s(1)])=k_Y([0,1])<\infty$.

  For Point~(ii), we deduce from~\eqref{eq:k_X-m_X} that
  \begin{equation*}
    \frac{dk_X}{d(\Lambda*s)}\leq\frac{C}{x},\ \forall x\in (0,s(\varepsilon)).
  \end{equation*}
  In view of of Prop.~\ref{prop:diff-kill-C1}, it only remains to prove that $d(\Lambda*s)/d\Lambda\leq C'$ in the neighborhood of 0.
  This is exactly~\eqref{eq:hyp-ds}.

  Finaly, under the conditions of Point~(iii), by~\eqref{eq:k_X-m_X}, $\kappa_X(y):=\frac{dk_X}{dm_X}(x)$ exists and is uniformly
  bounded on $(0,s(\varepsilon))$. Hence $X$ satisfies (C').
\end{proof}

Let us also mention that, as will appear clearly in the proofs of Section~\ref{sec:proof-results-kill}, our methods can be easily
extended to diffusion processes on a bounded interval, where one of the boundary point is an entrance boundary and the other is exit
or regular, and also to cases where both boundary points are either exit or regular.

We provide now a few examples of diffusion processes that enter our setting (directly or considering the above extension).

\bigskip\noindent \textbf{Example 1.} Let $X$ be a sticky Brownian motion on the interval $(-1,1)$ which is sticky at $0$, absorbed
at $-1$ and $1$, and with measurable killing rate $\kappa:(-1,1)\mapsto (0,+\infty)$ such that $\int_{(-1,1)} \kappa(x) dx<\infty$
and $\kappa(0)<\infty$. In this case $X$ is the diffusion with killing on natural scale with speed measure and killing measure
\begin{align*}
  m(dx)=\Lambda(dx)+\delta_0(dx)\quad\text{ and }\quad k(dx)=\kappa(x)\Lambda(dx)+\kappa(0)\delta_0(dx).
\end{align*}
Extending our result to diffusion processes on $(-1,1)$ with regular boundaries, we immediately deduce from
Proposition~\ref{prop:diff-kill-C2} that $X$ admits a unique quasi-stationary distribution and
satisfies~\eqref{eq:expo-cv-thm-with-killing}.

\bigskip\noindent \textbf{Example 2.} Let $Y$ be a solution to the SDE
\begin{align*}
  dY_t=\sqrt{Y_t}d B_t+Y_t (r-cY_t)dt,
\end{align*}
where $r\in\RR$ represents the individual growth rate without competition and $c>0$ governs the competitive density dependence in the
growth rate. We assume that the diffusion process $Y$ is subject to the killing rate $\kappa(y)=\sin(1/y)\vee \sqrt{y}$. Then $Y$
admits a unique quasi-stationary distribution and satisfies~\eqref{eq:expo-cv-thm-with-killing} (see Proposition~\ref{prop:logistic}
for a detailed study of this case in a more general setting).

In this example, $0$ is an exit boundary and $\kappa$ is not continuous at $0$ (see the comparison with the literature in
Section~\ref{sec:sturm-liouville}).

\bigskip\noindent \textbf{Example 3.} Let $Y$ be a solution to the SDE
\begin{align*}
  dY_t=dB_t-Y_t^2 dt,
\end{align*}
absorbed at $0$. We may for example assume that the diffusion process $Y$ is subject to the non-negative killing rate
$\kappa(y)=\frac{1}{\sqrt{y}}\vee \sqrt{y}$, or to the killing measure
\begin{align*}
  k(dy)=\sum_{n\in\N} b_n\delta_{a_n}(dy),
\end{align*}
for any bounded sequence $(a_n)_{n\in\N}$ and summable family $(b_n)_{n\in\N}$. In both cases, condition (D) is satisfied.

These two examples illustrate that our results also cover cases of unbounded (and even singular) killing rate $\kappa$.

\bigskip\noindent \textbf{Example 4.} In addition, our method is general enough to apply to more complex processes. An example of
diffusion without killing but with jumps is studied in~\cite[Sec.\,3.5.4]{champagnat-villemonais-15b}. The extension to the case of
diffusions with killing is straightforward using the same method: we consider a diffusion process on natural scale $(X_t,t\geq 0)$ on
$[0,\infty)$ with speed measure $m$ and killing measure $k$ such that either Condition~(C) or~(C') is satisfied. Let us denote by
$\mathcal{L}$ the infinitesimal generator of $X$.

Our first example is the Markov process $(\widehat{X}_t,t\geq 0)$ with infinitesimal generator 
$$
\widehat{\mathcal{L}}f(x)=\mathcal{L}f(x)+(f(x+1)-f(x))\mathbbm{1}_{x\geq 1},
$$
for all $f$ in the domain of $\mathcal{L}$. In other words, we consider a c\`adl\`ag process following a diffusion process with speed
measure $m$ and killing measure $k$ between jump times, which occur at the jump times of an independent Poisson process $(N_t,t\geq
0)$ of rate $1$, with jump size $+1$ if the process is above 1, and $0$ otherwise. Then, a straightforward adaptation of Prop.\ 3.10
of~\cite{champagnat-villemonais-15b} implies that the conclusion of Thm.~\ref{thm:QSD_full_with_killing} holds for $\widehat{X}$.

Another simple example is given by the Markov process $(\widehat{X}_t,t\geq 0)$ with infinitesimal generator 
$$
\widehat{\mathcal{L}}f(x)=\mathcal{L}f(x)+f(x+1)-f(x),
$$
for all $f$ in the domain of $\mathcal{L}$. Here, $+1$ jumps occur at Poisson times, regardless of the position of the process. Then,
in the case where $(m,k)$ satisfies Condition~(C) and if $\frac{dm}{d\Lambda}(x)\leq\frac{C}{x}$ in the neighborhood of $0$, the
conclusion of Thm.~\ref{thm:QSD_full_with_killing} holds for $\widehat{X}$.

To prove this, the only new difficulty compared with the previous case is to check that the probability that the process jumps before
being killed or hitting zero is smaller than $Cx$ when $\widehat{X}_0=x$. By Condition~(C), this is equivalent to prove the same
property for the process without killing. Since the first jump time is exponential of parameter 1, this can be proved with the exact
same argument as in Prop.~\ref{prop:diff-kill-C1} since $\tau_\d^c=\int_0^\infty L^y_{T^B_0}dm(y)$.

Of course, many easy extensions are possible, for example with jumps from $x$ to $g(x)$ at constant rate, where $g$ is a
non-decreasing function and might be $0$ on some interval. In fact, even random jumps can be easily covered provided monotonicity
properties of jump measures. This also covers situations with bounded non-decreasing jump rate. The case of general rates is more
complicated but could also be attacked with our method.

\subsection{Application to models of population dynamics}
\label{sec:popu-dyn}

Quasi-stationary distributions for one-dimensional diffusion processes have attracted much interest as an application to population
dynamics. In~\cite{CCLMMS09,cattiaux-meleard-10}, the authors consider the conditional behavior of diffusion processes conditioned
not to hit $0$. In~\cite{HeningKolb2014}, the authors consider one-dimensional diffusion processes that cannot hit $0$ continuously
($0$ is a natural boundary) but are subject to a killing rate modeling the risk of catastrophic events. In this section, our aim is
to bring together both points of view by considering logistic Feller diffusion processes and their extensions to general drifts, as
in~\cite{CCLMMS09}, with an addition killing rate.

More precisely, let us consider a diffusion process $Y$ on $[0,+\infty)$ solution to the SDE
\begin{align*}
  dY_t=\sqrt{Y_t}d B_t+Y_t h(Y_t)dt,
\end{align*}
where $h$ is a measurable function from $\RR_+\rightarrow\RR$ modeling the density dependence of the individual growth rate of the
population. We assume that $Y$ is subject to the nonnegative killing rate $\kappa=\frac{dk_Y}{dm_Y}\in
L^1_{\text{loc}}((0,+\infty))$. For such models, the conditions~(D) and~(D') directly give criteria for exponential convergence to
the quasi-distribution. The next result gives explicit conditions on $\kappa$ for a large class of functions $h$. Other conditions on
$\kappa$ could be obtained for other asymptotic behaviors of $h$.

\begin{prop}
  \label{prop:logistic}
  Assume that $h\in L^1_{\text{loc}}([0,+\infty))$ and that there exists $C,\beta>0$ such that, for some $y_0>0$,
  \begin{equation*}
    h(y)\leq -Cy^\beta, \quad\forall y\geq y_0.
  \end{equation*}
  If $\kappa$ satisfies
  \begin{equation*}
    \int_1^\infty\frac{\kappa(y)}{y^{1+\beta}}dy<\infty
  \end{equation*}
  and
  \begin{equation*}
    \text{either}\quad \int_0^1 \frac{\kappa(y)}{y}dy<\infty\quad\text{or}\quad \limsup_{x\rightarrow 0}\kappa(x)<\infty,
  \end{equation*}
  then there exists a unique probability measure $\alpha$ on $(0,+\infty)$ and two constants $C,\gamma>0$ such that, for all initial
  distribution $\mu$ on $[0,\infty)$ with $\mu(\{0\})<1$,
  \begin{align*}
    \left\|\PP_\mu(X_t\in\cdot\mid t<\tau_\partial)-\alpha(\cdot)\right\|_{TV}\leq C e^{-\gamma t},\ \forall t\geq 0.
  \end{align*}
\end{prop}

\begin{proof}
  In this case, the scale function and the speed measure associated with $Y$ are given by
  \begin{align*}
    s(y)=\int_0^y e^{-2\int_0^u h(z) dz} du \quad\text{and}\quad m_Y(dy)=\frac{e^{2\int_0^y h(u)du}}{y}dy.
  \end{align*}
  The killing measure associated to $Y$ is $k_Y=\kappa(y)m(dy)=\frac{\kappa(y)e^{2\int_0^y h(u)du}}{y}dy$. Our aim is to prove that
  condition (D) or (D') holds true using Proposition~\ref{prop:(D)-(D')}.

  To do so, let us first check that
  \begin{align}
    \label{eq:s-m+k-integrable}
    \int_{(0,+\infty)} s(y) (m_Y(dy)+k_Y(dy))<\infty.
  \end{align}
  We have
  \begin{align*}
    \int_{(0,+\infty)} s(y) (m_Y(dy)+k_Y(dy))&=\int_{(0,+\infty)} \int_0^y e^{-2\int_0^u h(z) dz} du \frac{(1+\kappa(y))e^{2\int_0^y h(u)du}}{y}dy.
  \end{align*}
  Since $h\in L^1_{\text{loc}}([0,+\infty))$, we have
  \begin{align*}
    \int_0^y e^{-2\int_0^u h(z) dz} du \frac{(1+\kappa(y))e^{2\int_0^y h(u)du}}{y}\sim_{y\rightarrow 0} 1+\kappa(y),
  \end{align*}
  which is integrable in any bounded neighborhood of $0$ by assumption. As a consequence,
  \begin{align*}
    \int_{(0,y_0)} s(y) (m_Y(dy)+k_Y(dy))<\infty.
  \end{align*}
  On the other hand, 
  \begin{align*}
    e^{2\int_0^y h(u)du}\,\int_0^y e^{-2\int_0^u h(z) dz} du&=\int_0^y e^{2\int_u^y h(z) dz} du \\
    &\leq \int_{0}^{y_0} e^{2\int_u^y h(z) dz} du+\int_{y_0}^y e^{-2C/(1+\beta)(y^{1+\beta}-u^{1+\beta})} du.
  \end{align*}
  We have, for all $y\geq y_0$,
  \begin{align*}
    \int_{0}^{y_0} e^{2\int_u^y h(z) dz} du\leq \int_{0}^{y_0} e^{2\int_u^{y_0} h(z) dz} du\,e^{-2C/(1+\beta)
      (y^{1+\beta}-y_0^{1+\beta})}=C' e^{-2Cy^{1+\beta}/(1+\beta)}
  \end{align*}
  for some finite constant $C'$. Moreover, setting $C_\beta=\sup_{x\in (0,1)} (1-x)/(1-x^{1+\beta})<\infty$, we have
  \begin{align*}
    \int_{y_0}^y e^{-2C/(1+\beta)(y^{1+\beta}-z^{1+\beta}) dz} du&\leq \int_{y_0}^y e^{-2C/(1+\beta)C^{-1}_\beta(y-z)y^\beta dz} du\\
    &\leq \frac{(1+\beta)C_\beta}{2C y^\beta}.
  \end{align*}
  As a consequence, there exists $C''>0$ such that
  \begin{align*}
    \int_{(y_0,+\infty)} s(y) (m_Y(dy)+k_Y(dy))&\leq C'' \int_{(y_0,+\infty)} \frac{1+\kappa(y)}{y^{1+\beta}} dy,
  \end{align*}
  which is finite by assumption. Finally, we conclude that~\eqref{eq:s-m+k-integrable} holds.
  
  The diffusion process $X=s(Y)$ is on natural scale with speed measure $m_X=m_Y*s$, which is dominated by $(2/x)\,d\Lambda(x)$ in the
  neighborhood of $0$, since $s(y)\sim_{0+} y$ and $s'(y)\rightarrow_{0+}1$. We deduce from Thm.~\ref{thm:h-transfo-3} that $X$
  satisfies~\eqref{eq:hyp-B-prime} and hence that $Y$ satisfies
  \begin{equation*}
    \PP(t_1<\tT_0\mid Y_0=y)\leq As(y),\quad\forall y>0.
  \end{equation*}

  On the one hand, if $\int_0^1 \frac{\kappa(y)}{y}dy<\infty$ then $\int_0^1 k_Y(dy)<\infty$, which allows us to conclude that (D)
  holds using Point (i) of Proposition~\ref{prop:(D)-(D')}. On the other hand, if $\limsup_{x\rightarrow 0}\kappa(x)<\infty$, then
  Point (iii) of the same proposition holds true and hence (D') holds.
\end{proof}

\subsection{On processes solutions to stochastic differential equations and comparison with the literature}
\label{sec:sturm-liouville}

In the case where the speed measure $m$ is absolutely continuous w.r.t.\ the Lebesgue measure on $(0,\infty)$, our diffusion
processes on natural scale are solutions, before killing, to SDEs of the form
\begin{equation}
  \label{eq:nos-EDS}
  dX_t=\sigma(X_t)dB_t,  
\end{equation}
where $\sigma$ is a measurable function from $(0,\infty)$ to itself such that the speed measure $m(dx)=\frac{1}{\sigma^2(x)}\,dx$ is
locally finite on $(0,\infty)$. Following the scale function trick of Section~\ref{sec:ex1}, our results actually cover all SDEs of
the form
\begin{equation*}
  dY_t=\sigma(Y_t)dB_t+b(Y_t)dt  
\end{equation*}
such that $b/\sigma^2\in L^1_{\text{loc}}((0,\infty))$ (see Chapter 23 of~\cite{kallenberg-02}).

Existence of quasi-stationary distribution of diffusion processes with killing has been already studied
in~\cite{Steinsaltz2007,kolb-steinsaltz-12}. These works are based on a careful study of the generator of the process and
Sturm-Liouville methods. Their results cover the case of entrance or natural boundary at $\infty$, but only the case of regular
boundary at 0. Our results only cover the case of entrance boundary at $\infty$, but $0$ can be either natural or exit. In the case
where $\infty$ is an entrance boundary, the stronger previously known result is~\cite[Thm.\,4.13]{kolb-steinsaltz-12} for diffusion
processes $(Y_t,t\geq 0)$ on $[0,\infty)$ which are solution to SDEs of the form
\begin{equation}
  \label{eq:EDS}
  dY_t=dB_t-h(Y_t)dt,
\end{equation}
with killing measure absolutely continuous w.r.t.\ the Lebesgue measure with continuous density on $[0,\infty)$. Let us recall that the
scale function of the solution $(Y_t,t\geq 0)$ to~\eqref{eq:EDS} is given by
$$
s(x)=\int_0^x\exp\left(2\int_1^yh(z)\,dz\right)\,dy,
$$
and that $X_t=s(Y_t)$ is solution to the SDE
\begin{equation}
  \label{eq:diffusion-ex}
  dX_t=\exp\left(2\int_1^{s^{-1}(X_t)}h(z)\,dz\right)\,dB_t.
\end{equation}
In particular, this diffusion with killing is not as general as~\eqref{eq:nos-EDS} since it only covers the case where the diffusion
coefficient is positive and continuous. In fact, in all the references previously cited, the drift $h$ is assumed at least continuous
so that the diffusion coefficient in~\eqref{eq:diffusion-ex} is at least $C^1$.

In addition, our criteria allow unbounded killing rates at $0$ as shown in the next proposition, contrary to existing results. This
result can be proved similarly as Proposition~\ref{prop:logistic}.

\begin{prop}
  \label{prop:brownien-drifte}
  Assume that $h\in L^1_{\text{loc}}([0,+\infty))$ and that there exists $C>0$ and $\beta>1$ such that, for some $y_0>0$,
  \begin{equation*}
    h(y)\leq -Cy^\beta, \quad\forall y\geq y_0.
  \end{equation*}
  If $\kappa$ satisfies
  \begin{equation*}
    \int_0^1 \kappa(y)dy<\infty  \quad    \text{and}\quad   \int_1^\infty\frac{\kappa(y)}{y^{\beta}}dy<\infty,
  \end{equation*}
  then (D) is satisfied.
\end{prop}

In the case where $\kappa$ is continuous bounded and $h$ is continuous on $[0,+\infty)$, existence of a quasi-stationary distribution
was already proved in~\cite{kolb-steinsaltz-12}. However, even in this case, uniqueness of the QSD, exponential convergence of
conditional distributions and exponential ergodicity of the $Q$-process are new. Conversely, other results
of~\cite{kolb-steinsaltz-12} cannot be easily studied with our methods, such as dichotomy results on the behavior of conditioned
diffusion processes with killing and the case of natural boundary at $\infty$.

\section{On continuous and discontinuous absorption times when $0$ is regular or exit and $\infty$ is entrance}
\label{sec:pties-diff}

We give in this section a result on the probability $\PP_x(\tau^c_\d<t\wedge\tau^d_\d)$ that the process hits $0$ continuously before
being killed, which is useful for the proof of our results of Subsection~\ref{thm:QSD_full_with_killing}. In the introduction, we called this
property \emph{coming down from infinity before killing} in reference to the classical property of coming down from infinity for
diffusion processes without killing (see~\cite{CCLMMS09,champagnat-villemonais-15b}). Since this result also has an interest by
itself, we give it in a separate section.

\begin{thm}
  \label{thm:pitman-yor}
  \begin{multline}
    \exists t>0\text{ such that }\lim_{x\rightarrow +\infty}\PP_x(\tau^c_\d<t\wedge\tau^d_\d)>0 \\
    \Longleftrightarrow \quad \int_0^\infty y\,(dk(y)+dm(y))<\infty.
    \label{eq:continuous-extinct-proba>0}
  \end{multline}
\end{thm}

This result is obtained as a consequence of the next proposition.

\begin{prop}
  \label{prop:exit-before-killing}
  For all $k$ and $m$ and all $t>0$,
  \begin{equation}
    \label{eq:exit-before-killing-1}
    \PP_x(\tau^c_\d<t\wedge\tau^d_\d)=\PP_x(\tau^c_\d<t)=\EE\left[\mathbbm{1}_{\int_0^\infty Z^x_s\,m(ds)<t}\exp\left(-\int_0^\infty
        Z^x_s\,k(ds)\right)\right],
  \end{equation}
  where
  $$
  Z^x_t=2\int_0^t\sqrt{Z^x_s}\,dW_s+2(t\wedge x),\quad\forall t\geq 0.
  $$
  In particular, $\PP_x(\tau^c_\d<t\wedge\tau^d_\d)$
  is non-increasing w.r.t.\ $x>0$ and
  \begin{equation}
    \label{eq:exit-before-killing-infinity}
    \lim_{x\rightarrow +\infty}\PP_x(\tau^c_\d<t\wedge\tau^d_\d)=\EE\left[\mathbbm{1}_{\int_0^\infty Z_s\,m(ds)<t}\exp\left(-\int_0^\infty
        Z_s\,k(ds)\right)\right],
  \end{equation}
  where $(Z_t,t\geq 0)$ is a squared Bessel process of dimension 2 started from 0.
\end{prop}

\begin{proof}
  We have
  \begin{align*}
    \PP_x(\tau^c_\d<t\wedge\tau^d_\d) & =\PP_x(\tT_0<t\text{ and }\kappa_{\tT_0}<\mathcal{E})\\
    &= \PP_x\left(\int_0^\infty L^y_{\sigma_{\tT_0}}\,m(dy)<t\text{ and }\int_0^\infty L^y_{\sigma_{\tT_0}}\,k(dy)<\mathcal{E}\right)\\ 
    & =\PP_x\left(\int_0^\infty L^y_{T^B_0}\,m(dy)<t\text{ and }\int_0^\infty L^y_{T^B_0}\,k(dy)<\mathcal{E}\right) \\ 
    & =\EE_x\left[\mathbbm{1}_{\int_0^\infty L^y_{T^B_0}\,m(dy)<t}\exp\left(-\int_0^\infty L^y_{T^B_0}\,k(dy)\right)\right].
  \end{align*}
  Hence~\eqref{eq:exit-before-killing-1} follows from Ray-Knight's theorem (cf.\ e.g.~\cite{rogers-williams-1-00}).

  Since the process $Y^x_t=\sqrt{Z^x_t}$ is solution to
  $$
  dY^x_t=dW_t+\frac{2\mathbbm{1}_{t\leq x}-1}{2Y^x_t}\,dt,
  $$
  standard comparison arguments show that the processes $(Y^x)_{x\geq 0}$ constructed with the same Brownian motion satisfy
  $Y^x_t\leq Y^{x'}_t$ a.s.\ for all $t\geq 0$ and $x\leq x'$, and $Y_t:=\lim_{x\rightarrow+\infty} Y^x_t$ is a Bessel process of
  dimension 2. Eq.~\eqref{eq:exit-before-killing-infinity} then follows from Lebesgue's theorem.
\end{proof}

\begin{proof}[Proof of Theorem~\ref{thm:pitman-yor}]
  In view of~\eqref{eq:exit-before-killing-infinity}, the equivalence~\eqref{eq:continuous-extinct-proba>0} follows
  from~\cite[Prop.\,(2.2),\,Lemma\,(2.3)]{pitman-yor-82}, which states that, for any square Bessel process $(R_t,t\geq 0)$ of
  positive dimension started from 0 and any positive Radon measure $\mu$ on $(0,+\infty)$,
  $$
  \PP\left(\int_{0}^\infty R_t\,d\mu(t)<\infty\right)=1 \quad\Longleftrightarrow\quad \int_0^\infty t\,d\mu(t)<\infty
  $$
  and
  $$
  \PP\left(\int_{0}^\infty R_t\,d\mu(t)=\infty\right)=1 \quad\Longleftrightarrow\quad \int_0^\infty t\,d\mu(t)=\infty.
  $$
  Hence both r.v.\ $\int_0^\infty L^y_{T^B_0}\,m(dy)$ and $\int_0^\infty L^y_{T^B_0}\,k(dy)$ are a.s.\ finite iff $\int_0^\infty
  y\,(dk(y)+dm(y))<\infty$, which entails~\eqref{eq:continuous-extinct-proba>0}.
\end{proof}

\section{Proof of the results of Section~\ref{sec:QSD-diff-kill}}
\label{sec:proof-results-kill}

Theorems~\ref{thm:QSD_full_with_killing} and~\ref{thm:Q-proc-with-killing} and the main part of
Proposition~\ref{prop:eta-with-killing} directly follow from the results on general Markov processes
of~\cite{champagnat-villemonais-15}. More precisely, the following condition (A) is equivalent to~\eqref{eq:expo-cv-thm-with-killing}
(\cite[Thm.\,2.1]{champagnat-villemonais-15}), and implies properties~\eqref{eq:convergence-to-eta-with-killing}
and~\eqref{eq:eigen-function-with-killing} of Proposition~\ref{prop:eta-with-killing} (\cite[Prop.\,2.3]{champagnat-villemonais-15})
and the whole Theorem~\ref{thm:Q-proc-with-killing} (\cite[Thm.\,3.1]{champagnat-villemonais-15}).

\paragraph{Assumption~(A)}
There exists a probability measure $\nu$ on $(0,+\infty)$ such that
\begin{itemize}
\item[(A1)] there exists $t_0,c_1>0$ such that for all $x>0$,
  $$
  \PP_x(X_{t_0}\in\cdot\mid t_0<\tau_\partial)\geq c_1\nu(\cdot);
  $$
\item[(A2)] there exists $c_2>0$ such that for all $x>0$ and $t\geq 0$,
  $$
  \PP_\nu(t<\tau_\partial)\geq c_2\PP_x(t<\tau_\partial).
  $$
\end{itemize}

Hence, we need first to prove that (C) implies (A) (in Subsection~\ref{sec:(C)->(A)}), second, to prove that (C') implies (A) (in
Subsection~\ref{sec:(C')->(A)}), and finally, to prove that $\eta(x)\leq Cx$ for all $x\geq 0$ (in
Subsection~\ref{sec:proof-eta-with-killing}), which is the only part of Proposition~\ref{prop:eta-with-killing} left to prove.

\subsection{Proof that Condition~(C) implies (A)}
\label{sec:(C)->(A)}

Note that~\cite[Thm.\,2.1]{champagnat-villemonais-15} also assumes that
\begin{align}
  \label{eq:step0}
  \P_x(t<\tau_\d)>0,\quad \forall x>0,\quad \forall t>0,
\end{align}
which is entailed by Prop.~\ref{prop:regularity}.

Our proof of~(A) follows four steps, similarly as the proof of~\cite[Thm.\,3.1]{champagnat-villemonais-15b}. In the first step, we
prove that when $X_0$ is close to $0$, then, conditionally on non-absorption, the process exits some neighborhood of $0$ with
positive probability in bounded time. In the second step, we construct the measure $\nu$ involved in~(A). We next prove~(A1) in the
third step, and~(A2) in the last step.

\medskip\noindent  \textit{Step 1: the conditioned process escapes a neighborhood of 0 in finite time.}\\
The goal of this step is to prove that there exists $\varepsilon,c>0$ such that
\begin{equation}
  \label{eq:step1_killing}
  \P_x(X_{t_1}\geq\varepsilon\mid t_1<\tau_\d)\geq c,\quad\forall x>0, 
\end{equation}
where $t_1$ is taken from Assumption~(C).

To prove this, recall that $\widetilde{X}$ is a diffusion process on natural scale with speed measure $m$ but with null killing
measure. We first observe that, since $\widetilde{X}$ is a local martingale, for all $x\in(0,1)$,
\begin{align*}
  x&=\E_x(\widetilde{X}_{t_1\wedge \widetilde{T}_1\wedge \tau_\kappa}),
\end{align*}
where we recall that $\tau_\kappa=\tau^d_\d$ on the event $\tau^d_\d<\infty$ and $\widetilde{T}_0=\tau_\d^c$ on the event
$\tau^c_\d<\infty$. But the absorption at $0$ ensures that
\begin{align*}
  \widetilde{X}_{t_1\wedge \widetilde{T}_1\wedge \tau_\kappa}=\widetilde{X}_{t_1\wedge \widetilde{T}_1\wedge
    \tau_\kappa}\11_{t_1\wedge \tau_\kappa< \widetilde{T}_0}+\11_{\tT_1< \widetilde{T}_0 < t_1\wedge \tau_\kappa}.
\end{align*}
We thus have for all $x\in (0,1)$
\begin{align*}
  x=\P_x(t_1\wedge \tau_\kappa<\widetilde{T}_0)\E_x(\widetilde{X}_{t_1\wedge \widetilde{T}_1\wedge \tau_\kappa}\mid t_1\wedge
  \tau_\kappa< \widetilde{T}_0) +\P_x(\tT_1< \widetilde{T}_0 < t_1\wedge \tau_\kappa).
\end{align*}
Now, the strong Markov property entails
\begin{align*}
  \P_x(\tT_1<\tT_0<t_1\wedge\tau_\kappa)\leq \P_x(\tT_1<\tT_0)\P_1(\tT_0<t_1\wedge\tau_\kappa).
\end{align*}
Since $\P_1(\tT_0<t_1\wedge\tau_\kappa)<1$ (see Prop.~\ref{prop:exit-before-killing}) and $\P_x(\tT_1\leq \tT_0)=x$ (recall that
$\tX$ is a local martingale), we deduce from the two previous equations that there exists a constant $A'>0$ such that
\begin{align*}
  A'x\leq \P_x(t_1\wedge \tau_\kappa<\widetilde{T}_0)\E_x(\widetilde{X}_{t_1\wedge \widetilde{T}_1\wedge \tau_\kappa}\mid t_1\wedge
  \tau_\kappa< \widetilde{T}_0).
\end{align*}
By Assumption~(C), we have $\P_x(t_1<\tT_0)\leq Ax$ and $\P_x(\tau_\d^d<\tau_\d^c)\leq Ax$. But, by definition of $\tau_\d^c$ and
$\tau_\d^d$, we have $\{\tau_\d^d<\tau_\d^c\}=\{\tau_\kappa<\tT_0\}$, so that
\begin{align*}
  \PP_x(t_1\wedge\tau_\kappa<\tT_0)\leq 2Ax.
\end{align*}
As a consequence,
\begin{align*}
  \E_x\left(1-\tX_{t_1\wedge \tT_1\wedge \tau_\kappa}\ \left|\ t_1\wedge \tau_\kappa<\tT_0\right.\right)\leq 1-\frac{A'}{2A}.
\end{align*}
Since we can assume without loss of generality that $\frac{A'}{2A}<1$, Markov's inequality implies that, setting
$b=1-\sqrt{1-A'/2A}$,
\begin{equation*}
  \P_x\left(\widetilde{X}_{t_1\wedge \tT_1\wedge \tau_\kappa}\leq b\ \left|\ t_1\wedge\tau_\kappa<\tT_0\right.\right)\leq \frac{1-A'/2A}{1-b}=1-b,
\end{equation*}
hence
\begin{align*}
  \P_x\left(\widetilde{X}_{t_1\wedge \tT_1\wedge \tau_\kappa}> b\ \left|\ t_1\wedge\tau_\kappa<\tT_0\right.\right)\geq b.
\end{align*}
We deduce that, for all $x\in(0,b)$,
\begin{align}
  \label{eq:calcul_killing}
  \P_x(\tT_b<t_1\wedge \tau_\kappa\mid t_1\wedge \tau_\kappa<\tT_0)=\P_x(\tT_b<t_1\wedge \tT_1\wedge \tau_\kappa\mid
  t_1\wedge \tau_\kappa<\tT_0)\geq b.
\end{align}
Now, since $\P_b(t_1<\tT_0\wedge \tau_\kappa)>0$ (see Prop.~\ref{prop:regularity}), there exists $\varepsilon\in(0,b)$ such that
\begin{equation}
  \label{eq:def-epsilon_killing}
  \P_{b}(t_1< \tT_\varepsilon\wedge \tau_\kappa)>0.    
\end{equation}
Hence, we deduce from the strong Markov property that
\begin{align*}
  \P_x(X_{t_1}\geq\varepsilon) &=\P_x(\tX_{t_1}\geq \varepsilon\text{ and }t_1<\tau_\kappa)\\
  & \geq\P_x(\tT_{b}<t_1\wedge \tau_\kappa)\P_{b}(t_1<\tT_\varepsilon\wedge\tau_\kappa) \\
  & \geq b\,\P_x(t_1\wedge\tau_\kappa<\tT_0)\,\P_{b}(t_1<\tT_\varepsilon\wedge \tau_\kappa),
\end{align*}
where the last inequality follows from~\eqref{eq:calcul_killing}. But $\P_x(t_1\wedge\tau_\kappa<\tT_0)\geq
\P_x(t_1<\tT_0\wedge\tau_\kappa)=\P_x(t_1<\tau_\d)$, so that
\begin{align*}
  \P_x(X_{t_1}\geq\varepsilon\mid t_1<\tau_\d)\geq b\,\P_{b}(t_1<\tT_\varepsilon\wedge \tau_\kappa)>0.
\end{align*}
This entails~\eqref{eq:step1_killing} for $x<b$.

For $x\geq b$, the continuity (before killing) and the strong Markov property for $X$ imply
\begin{align*}
  \P_x(X_{t_1}>\varepsilon\mid t_1<\tau_\d)&\geq\P_x(X_{t_1}>\varepsilon)\\
  &\geq\P_x(T_b<\infty)\P_{b}(t_1<\tT_\varepsilon\wedge\tau_\kappa)\\
  &\geq \P_x(\tau_\d^c<\infty)\P_{b}(t_1<\tT_\varepsilon\wedge\tau_\kappa)
\end{align*}
Using Proposition~\ref{prop:exit-before-killing} and~\eqref{eq:def-epsilon_killing}, the proof of~\eqref{eq:step1_killing} is completed.

\medskip \noindent \textit{Step 2: Construction of coupling measures for the unconditioned process.}\\
We prove that there exist two constants $t_2,c'>0$ such that, for all $x\geq\varepsilon$,
\begin{equation}
  \label{eq:step2_killing}
  \P_x(X_{t_2}\in \cdot)\geq c'\nu,
\end{equation}
where 
$$
\nu=\P_\varepsilon(X_{t_2}\in\cdot\mid t_2<\tau_\d).
$$

Fix $x\geq\varepsilon$ and construct two independent diffusions $X^\varepsilon$ and $X^x$ with speed measure $m(dx)$, killing measure
$k(dx)$, and initial value $\varepsilon$ and $x$ respectively. Let $\theta=\inf\{t\geq 0:X^\varepsilon_t=X^x_t\}$. By the strong
Markov property, the process
$$
Y^x_t=
\begin{cases}
  X^x_t & \text{if }t\leq\theta, \\
  X^\varepsilon_t & \text{if }t>\theta
\end{cases}
$$
has the same law as $X^x$. By the continuity of the paths of the diffusions before killing, $\theta\wedge
\tau_\d^{d,x}\leq\tau^x_\d:=\inf\{t\geq 0: X^x_t=0\}$. Hence, for all $t>0$,
$$
\P(\theta<t)\geq\P(\tau^{c,x}_\d<t\text{ and }\tau_\d^{c,x}<\tau^{d,x}_\d).
$$
By Theorem~\ref{thm:pitman-yor}, there exists $t_2>0$ and $c''>0$ such that
\begin{align*}
  \inf_{y>0}\P_y(\tau^{c}_\d<t_2\text{ and }\tau_\d^{c}<\tau^{d}_\d)\geq c''>0.
\end{align*}
Hence
\begin{align*}
  \P_x(X_{t_2}\in\cdot)&=\P(Y^x_{t_2}\in\cdot)\geq\P(X^\varepsilon_{t_2}\in\cdot,\ \tau^{c,x}_\d<t_2\text{ and }\tau_\d^{c,x}<\tau^{d,x}_\d)\\
  &\geq c''\P_\varepsilon(X_{t_2}\in \cdot),
\end{align*}
where the last inequality follows from the independence of $X^x$ and $X^\varepsilon$. Therefore, \eqref{eq:step2_killing} is proved
with $c'=c''\P_\varepsilon(t_2<\tau_\d)$.

\medskip \noindent \textit{Step 3: Proof of (A1).}\\
Using successively the Markov property, Step 2 and Step 1, we have for all $x>0$
\begin{align*}
  \P_x(X_{t_1+t_2}\in\cdot\mid t_1+t_2<\tau_\d) & \geq \P_x(X_{t_1+t_2}\in\cdot\mid t_1<\tau_\d) \\ 
  & \geq \int_\varepsilon^\infty \P_y(X_{t_2}\in\cdot)\P_x(X_{t_1}\in dy\mid t_1<\tau_\d) \\
  & \geq c'\nu(\cdot)\P_x(X_{t_1}\geq\varepsilon\mid t_1<\tau_\d)\\
  & \geq c\,c' \,\nu(\cdot).
\end{align*}
This entails (A1) with $t_0=t_1+t_2$ and $c_1=cc'$.

\medskip \noindent \textit{Step 4: Proof of (A2).}\\ 
Let $a>0$ be such that $\nu([a,+\infty))>0$. Then we have, for all $x\geq a$,
\begin{align*}
  \P_x(X_{t_0}\geq a)&\geq c_1\nu([a,+\infty))\P_x(\tau_\d> t_0)\\
  &\geq c_1\nu([a,+\infty)) \P_x(T_a<\infty)\P_a(\tau_\d>t_0)\\
  &\geq c_1\nu([a,+\infty)) \P_x(\tau_\d^c < \tau_\d^d)\P_a(\tau_\d> t_0),
\end{align*}
where we have used the strong Markov property for the second inequality. By Theorem~\ref{thm:pitman-yor}, $\inf_{x\in(0,+\infty)}
\P_x(\tau_\d^c < \tau_\d^d)>0$ and we know that $\P_a(\tau_\d> t_0)>0$ by Proposition~\ref{prop:regularity}. As a consequence,
\begin{align*}
  \inf_{x\geq a} \P_x(X_{t_0}\geq a) >0.
\end{align*}
Using this inequality and a standard renewal argument, we deduce that there exists $\rho>0$ such that, for all $k\in\N$,
\begin{align}
  \label{eq:lemme1_killing}
  \P_a(X_{kt_0}\geq a)\geq e^{-\rho k t_0}.
\end{align}
Now, we know from Step~4 of the proof of \cite[Theorem~3.1]{champagnat-villemonais-15b} (see also~\cite{CCLMMS09}) that, since the
diffusion $\tX$ on $[0,\infty)$ without killing has $\infty$ as an entrance boundary, for all $r>0$, there exists $y_r>0$ (which can
be assumed larger than $a$ without loss of generality) such that
\begin{equation*}
  \sup_{x\geq y_r}\E_x(e^{r \tT_{y_r}})<+\infty.
\end{equation*}
This implies
\begin{equation}
  \label{eq:ineg0_killing}  
  \sup_{x\geq y_\rho}\E_x(e^{\rho (T_{y_\rho}\wedge \tau_\d)})<+\infty,
\end{equation}
for the constant $\rho$ of~\eqref{eq:lemme1_killing}. Using the strong Markov property, for all $x\in[a,+\infty)$ and
$y\in[a,y_\rho]$,
\begin{align*}
  \P_x(t<\tau_\d)&\geq \P_x(T_a<\infty)\P_a(T_y<\infty)\P_y(t<\tau_\d)\\
  &\geq \P_x(\tau_\d^c<\tau_\d^d)\P_a(T_{y_\rho}<\infty)\P_y(t<\tau_\d).
\end{align*}
Using Theorem~\ref{thm:pitman-yor} and Proposition~\ref{prop:regularity}, we infer that there exists $C>0$ such that, for all $t\geq
0$,
\begin{align}
  \label{eq:ineg1_killing}
  \sup_{x\in[a,y_\rho]} \P_x(t<\tau_\d) \leq C\inf_{x\in [a,+\infty)}\P_x(t<\tau_\d).
\end{align}
Finally, we also deduce from the Markov property that, for all $s<t$,
\begin{align}
  \label{eq:ineq2_killing}
  \P_a(X_{\lceil s/t_0 \rceil t_0}\geq a)\inf_{x\in[a,+\infty)}\P_x(t-s<\tau_\d)\leq \P_a(t<\tau_\d).
\end{align}

Now, for all $x\geq y_\rho$, with a constant $C>0$ that may change from line to line, using
successively~\eqref{eq:ineg0_killing},~\eqref{eq:ineg1_killing}, \eqref{eq:ineq2_killing} and~\eqref{eq:lemme1_killing}, we obtain
\begin{align*}
  \P_x(t<\tau_\d)&\leq \P_x(t< T_{y_\rho}\wedge\tau_\d)+\int_{0}^t \P_{y_\rho}(t-s<\tau_\d)\,\P_x(T_{y_\rho}\in ds)\\
  &\leq Ce^{-\rho t}+C\int_0^t \P_a(t-s<\tau_\d)\,\P_x(T_{y_\rho}\in ds)\\
  &\leq Ce^{-\rho \lceil t/t_0 \rceil t_0}+C \P_a(t<\tau_\d)\int_0^t \frac{1}{\P_a(X_{\lceil s/t_0\rceil t_0}\geq a)}\,\P_x(T_{y_\rho}\in ds)\\
  &\leq C\P_a(t<\tau_\d)+C \P_a(t<\tau_\d)\int_0^t e^{\rho s}\,\P_x(T_{y_\rho}\in ds).
\end{align*}
We deduce that, for all $t\geq 0$,
\begin{align}
  \label{eq:findufin}
  \sup_{x\in [y_{\rho},+\infty)}\P_x(t<\tau_\d)\leq C \inf_{x\in[a,+\infty)}\P_x(t<\tau_\d).
\end{align}
Now, for all $x\in[a,+\infty)$ and all $y\in(0,a)$,
\begin{align*}
  \P_x(t<\tau_\d)&\geq \P_x(T_y<\infty)\P_y(t<\tau_\d)\\
  &\geq \P_x(\tau_\d^c<\tau_\d^d)\P_y(t<\tau_\d).
\end{align*}
Using Proposition~\ref{prop:exit-before-killing}, we deduce that there exists $C>0$ such that, $\forall t\geq 0$,
\begin{align}
  \label{eq:findufin2}
  \sup_{x \in(0,a)}\P_x(t<\tau_\d)\leq C \inf_{x\in[a,+\infty)}\P_x(t<\tau_\d).
\end{align}
Finally, since $\nu([a,+\infty))>0$, \eqref{eq:ineg1_killing}, \eqref{eq:findufin} and~\eqref{eq:findufin2} entail (A2).

\subsection{Proof that Condition~(C') implies (A)}
\label{sec:(C')->(A)}

The proof follows exactly the same steps as in Subsection~\ref{sec:(C)->(A)}. The only step which needs to be modified is Step~1, so
we only detail this step.

\medskip\noindent  \textit{Step 1: the conditioned process escapes a neighborhood of 0 in finite time.}\\
The goal of this step is to prove that there exists $\varepsilon',c>0$ such that
\begin{equation}
  \label{eq:step1_killing-bis}
  \P_x(X_{t_1}\geq\varepsilon'\mid t_1<\tau_\d)\geq c,\quad\forall x>0.    
\end{equation}
We know from the study of diffusions without killing of~\cite[Step\,1\,of\ Section\,5.1]{champagnat-villemonais-15b} that there
exists $c,b>0$ such that, for all $x\leq b$,
\begin{equation}
  \label{eq:last-proof-1}
  \PP_x(\tT_{b}<t_1\wedge\tT_0)\geq c\,\PP_x(t_1<\tT_0),    
\end{equation}
By Assumption~(C') and using~\eqref{eq:As}, a.s.\ for all $t\leq \tT_\varepsilon$,
$$
\kappa_t=\int_0^\infty L^y_{\sigma_t}\,dk(y)\leq A \int_0^\infty L^y_{\sigma_t}\,dm(y)=A\, A_{\sigma_t}=At.
$$
Note that we can assume without loss of generality in~\eqref{eq:last-proof-1} that $b\leq\varepsilon$. Therefore, for all
$x\in(0,b)$, a.s.\ under $\PP_x$,
$$
\{\tau_\kappa>t_1\wedge\tT_b\}=\{\kappa_{t_1\wedge\tT_b}<\cE\}\supset\{t_1\wedge\tT_b<\cE/A\}\supset\{t_1<\cE/A\}.
$$
Hence, using~\eqref{eq:last-proof-1} and the independence between $\cE$ and $\tX$,
\begin{align*}
  \PP_x(T_b<t_1\wedge\tau_\d) & =\PP_x(\{\tT_b<t_1\wedge\tT_0\}\cap\{\tau_\kappa>t_1\wedge\tT_b\}) \\
  & \geq \PP_x(\{\tT_b<t_1\wedge\tT_0\}\cap\{t_1<\cE/A\}) \\
  & \geq c\PP_x(t_1<\tT_0)e^{-At_1}\geq  ce^{-At_1}\PP_x(t_1<\tau_\d).
\end{align*}
But, for all $\varepsilon'\in(0,b),x\in(0,b)$,
$$
\PP_x(X_{t_1}\geq\varepsilon')\geq\PP_x(T_b<t_1\wedge\tau_\d)\PP_b(t_1<T_{\varepsilon'}\wedge\tau_\kappa).
$$
By Proposition~\ref{prop:regularity}, the last factor of the r.h.s.\ is positive for $\varepsilon'>0$ small enough, which concludes
the proof of~\eqref{eq:step1_killing-bis}, since the case $x\geq b$ can be handled exactly as in Subsection~\ref{sec:(C)->(A)}.

\subsection{Proof that $\eta(x)\leq Cx$ for all $x\geq 0$}
\label{sec:proof-eta-with-killing}

We recall that, since we already proved that (C) and (C') imply (A),~\cite[Prop.\,2.3]{champagnat-villemonais-15} entails the main
part of Proposition~\ref{prop:eta-with-killing}, and we only have to check that there exists a constant $C$ such that $\eta(x)\leq C
x$ for all $x\geq 0$, where $\eta(x)$ is the uniform limit of $e^{\lambda_0 t}\PP_x(t<\tau_\d)$. Since for all $t\geq t_1$, where the
constant $t_1$ is the one given in Assumption~(C) or~(C'),
\begin{equation}
  \label{eq:last-prop}
  e^{\lambda_0 t}\PP_x(t<\tau_\d)=e^{\lambda_0 t}\EE_x\left[\PP_{X_{t_1}}(t-t_1<\tau_\d)\right]\leq Ce^{\lambda_0 t_1}\PP_x(t_1<\tau_\d).
\end{equation}
For the last inequality, we used the fact that $e^{\lambda_0 s}\PP_y(s<\tau_\d)$ is uniformly bounded in $y>0$ and $s\geq 0$ because of
the uniform convergence in~\eqref{eq:convergence-to-eta-with-killing}.

By Assumption~(C) or (C'), $\PP_x(t_1<\tau_\d)\leq \PP_x(t_1<\tT_0)\leq Ax$, and we obtain the inequality $\eta(x)\leq Cx$ by letting
$t\rightarrow+\infty$ in~\eqref{eq:last-prop}.

\def\cprime{$'$} \def\cprime{$'$} \def\cprime{$'$} \def\cprime{$'$}
  \def\polhk#1{\setbox0=\hbox{#1}{\ooalign{\hidewidth
  \lower1.5ex\hbox{`}\hidewidth\crcr\unhbox0}}} \def\cprime{$'$}
  \def\lfhook#1{\setbox0=\hbox{#1}{\ooalign{\hidewidth
  \lower1.5ex\hbox{'}\hidewidth\crcr\unhbox0}}} \def\cprime{$'$}
  \def\cprime{$'$}

\end{document}